\documentclass[12pt, oneside, reqno]{amsart}
\usepackage[dvipsnames]{xcolor}
\usepackage{amsfonts,amsmath,mathabx,fullpage,amssymb,amsthm,tikz,mathtools,lmodern,paralist,mathrsfs,hyperref,stmaryrd,enumitem,tikz-cd,ytableau,tcolorbox,cases, soul, accents}
\usepackage[normalem]{ulem}
\hypersetup{colorlinks, citecolor=red, filecolor=black, linkcolor=blue, urlcolor=blue}
\usetikzlibrary{calc,3d,patterns}

\usepackage{color}
\usepackage[color]{xy}
\colorlet{cyan}[rgb]{cyan} 
\ytableausetup{smalltableaux, aligntableaux=center}

\usepackage[ruled]{algorithm2e}
\usepackage{booktabs}
\usepackage{tabularx}
\usepackage{tabu}

\newcommand{\excise}[1]{}

\newtheorem{thm}{Theorem}[section]

\newtheorem{lemma}[thm]{Lemma}

\newtheorem{defn}[thm]{Definition}
\newtheorem{cor}[thm]{Corollary}

\newtheorem{prob}[thm]{Open Problem}

\theoremstyle{definition}
\newtheorem{rem}[thm]{Remark}

\newcommand{\Symm}{Sym}
\newcommand{\Sym}{Sym_5}
\newcommand{\Ss}{\mathcal{S}}
\newcommand{\CPP}{\mathcal{CP}}
\newcommand{\CP}{\mathcal{CP}_5}
\newcommand{\DNN}{\mathcal{DNN}}
\newcommand{\DN}{\mathcal{DNN}_5}
\newcommand{\NN}{\mathcal{N}}
\newcommand{\COPP}{\mathcal{COP}}
\newcommand{\COP}{\mathcal{COP}_5}
\newcommand{\CRF}{\mathcal{CPR}_5(5)}

\newcommand{\CR}{\mathcal{CPR}_n(k)}
\newcommand{\RF}{\mathbb{R}^5_{\ge 0}}
\newcommand{\cone}{\text{cone}}

\begin{document}
\title{The cone of $5\times 5$ completely positive matrices}
\author {Max Pfeffer}
\address{Technische Universit\"at Chemnitz}
\email{max.pfeffer@math.tu-chemnitz.de}
\address{Departamento de Matem\'aticas \\ Pontificia Universidad Cat\'olica de Chile}
\email{jsamper@mat.uc.cl}\author{Jos\'e Alejandro Samper}

\date{\today}

\maketitle

\begin{abstract}
We study the cone of completely positive (cp) matrices for the first interesting case $n = 5$. This is a semialgebraic set for which the polynomial equalities and inequlities that define its boundary can be derived. We characterize the different loci of this boundary and we examine the two open sets with cp-rank 5 or 6. A numerical algorithm is presented that is fast and able to compute the cp-factorization even for matrices in the boundary. With our results, many new example cases can be produced and several insightful numerical experiments are performed that illustrate the difficulty of the cp-factorization problem.
\end{abstract}

\section{Introduction}
Conic optimization is the problem of minimizing the value of a linear function over the intersection of a cone and a linear space. Many problems in optimization and geometry can be framed in this form and a wide variety of minimization methods have been developed for different types of cones. 
Cones that are also semialgebraic sets are of particular interest, because their boundary can be described by polynomial inequalities. Special relevant cases include polyhedral cones, the cone of positive semidefinite matrices, the cone of homogeneous nonnegative polynomials in any number of variables, or the cone of homogeneous polynomials that are sums of squares. 

There are many natural questions that may be asked about such semialgebraic cones. In order to asses the effectiveness of potential optimization algorithms, we need measures of the "complicatedness" of the cone. For instance, it is often hard to determine if a given vector is in a specific cone (the membership problem), and one may wonder if there are deep reasons for this. One potential approach consists of understanding the system of polynomial inequalities defining the cone. Then the number of monomials and their degree can serve as a measure for this hardness. Particularly interesting is the cone of (homogeneous) sums of squares, which is contained in the cone of nonnegative polynomials, but the algebraic description of their differences is difficult even in the smallest interesting cases, see \cite{Blekherman12, BlekhermanEtAl2012}.

In this article, we aim to carry out such an analysis for a more complicated cone that is also of interest in optimization: the collection of completely positive real symmetric $n\times n$ matrices $\CPP_n$, i.e., real nonnegative $n\times n$ matrices $A$ that can be written as $A=BB^{\mathsf T}$, where $B$ is also a nonnegative matrix. This convex cone is dual to the cone $\COPP_n$ of symmetric matrices whose associated quadratic form is nonnegative in the nonnegative orthant.  
The cones $\CPP_n$ and $\COPP_n$ are easily shown to be semialgebraic, but they are notably complicated to work with. For instance, determining whether a given matrix is in $\CPP_n$ is a co-NP-complete problem \cite{MurtyKabadi87}. There are several algorithms that attempt to factorize these matrices in order to test if they are in $\CPP_n$, but these algorithms (both exact and approximate) have a well established number of drawbacks: exact factorization algorithms are slow and/or fail to detect matrices near the boundary of the cone, while approximation algorithms are not reliable for matrices near the boundary of the cone. 

To try to understand the difficulty, one possibility is to study low dimensional cases in detail. For $n=1,2,3,4$, the cone $\CPP_n$ coincides with the cone $\DNN_n$ of matrices that are positive semidefinite and have nonnegative entries. Starting at $n=5$, the straightforward containment $\CPP_n\subseteq \DNN_n$ starts being strict. The boundary of $\DNN_n$ is generally easy to describe (matrices of low rank and/or with some zero entries), and it is hence not hard to describe which part of that boundary is also in $\CPP_n$. The interesting part is then to describe the boundary of $\CPP_n$ cointained in the interior of $\DNN_n$. Additionally, there is a partition of $\CPP_n$ according to what is called the completely postive rank (cp-rank), dealing with the smallest size of a nonnegative matrix factorization. The boundaries of the regions are also semialgebraic and in general quite difficult: it is not known how many parts are in this splitting.

In this work, we thoroughly investigate the smallest interesting case, namely $n=5$. We focus on the boundary $\partial \CP$ from an algebraic point of view, finding explicit equations for part of the boundary and implicit ones for the rest. We show that the Zariski-closure of the part of the boundary of $\CP$ in the interior of $\DN$ is a degree 3900 hypersurface with 24 irreducible components, 12 of which have degree 320 and 12 degree 5. The explicit polynomials for the degree 320 parts are probably impossible to compute, but the parametrization comes from a simple toric variety, which yields the correct scenario for the novel implicitization techniques form numerical algebraic geometry \cite{ChenKileel}. 

Additionally, the algebraic description of the boundary allows us to show an interesting fact and produce a number of computational experiments: we obtain that the set of matrices with rational entries is dense in the boundary $\partial \CP$. It also yields a recipe to construct many exact examples in all the components of the boundary as well as computing the tangent space at any given point. This data allows us to find matrices in $\DN\setminus \CP$ together with their closest point in $\partial\CP$. We include a discussion of the cp-rank of the matrices in the interior of $\CP$. The possible cp-ranks are known to be 5 or 6 and the boundary is again an algebraic surface. Although not much is known about this boundary, we highlight some of its properties in order to pursue some computational experiments. 

Next we present a novel numerical method for the approximation of the cp-factorization. This method is very fast and it is able to approximate factorizations even of matrices in the boundary of $\CP$. We carry out a number of experiments to estimate its performance in the small-dimensional setting $n = 5$. For instance, in the part of the boundary that does not coincide with the boundary of $\DNN_5$, the factorizations of the matrices are forced to have some zeros. The algorithm easily detects these zero entries in the factorization and generally finds the correct factorization. Using this, and with additional knowledge about the boundary separating matrices of cp-rank 5 and 6, the results of the experiments allow us to formulate a couple of relevant questions and conjectures. 

In short, the strength of this paper is that it combines the theoretical progress (the derivation of the algebraic equations for the boundary) with an experimental investigation of some interesting cases. Since the completely positive cone is of great importance in optimization, it is helpful to obtain this practical insight.

\subsection{Notation}
We will consider several convex cones contained in the space of real symmetric $n\times n$ matrices. The following cones will be relevant: 
\begin{itemize}
\item $\Symm_n$ denotes the space of $n\times n$ real symmetric matrices. 
\item $\Ss_n$ is the cone of all positive semidefinite matrices, i.e., matrices $M$ that can be written as $M=XX^{\mathsf T}$ for some $n\times k$ matrix $X$. 
\item $\NN_n$ is the cone of all symmetric matrices with nonnegative entries.
\item $\DNN_n= \Ss_n\cap \NN_n$ denotes the cone of all positive semidefinite matrices with nonnegative entries. This is sometimes called the {\em doubly nonnegative cone}. 
\item $\COPP_n$ is the {\em copositive cone} of all matrices $M$ such that $v^{\mathsf T}Mv \ge 0$ for all $v \in \mathbb{R}_{\ge 0}^n$. 
\item $\CPP_n$ is the cone of all matrices $A$ such that there is a $n\times k$ matrix $B$ with nonnegative entries and $A=BB^{\mathsf T}$. This cone is called the \emph{completely positive cone}.
\end{itemize}
The cones are semialgebraic sets, meaning that they can be described by polynomial inequalities. We are interested in understanding the difference between the cones $\CPP_n\subseteq \DNN_n$ as good as we can. Maxfield and Minc \cite{MaxfieldMinc1963} showed that the two cones are equal if and only if $n\le 4$. Thus we will focus in understanding the case $n=5$, i.e., the smallest value of $n$ in which the two cones are different. A particularly interesting question is to understand the subset $\partial\CP \cap \DN^\circ$, i.e., the elements in the boundary of $\CP$ that lie in the interior of $\DN$. 

Endow the space of symmetric matrices with the inner product $\langle A,B \rangle = \text{trace}(A^{\mathsf T}B)$ and the corresponding Frobenius norm $\| A \| = \sqrt{\langle A,A \rangle}$. Then $\CPP$ and $\COPP$ are dual cones in this setting. This will be exploited when we study the boundary of the cones. 

\section{The boundary of $\CP$}

We will use the extreme rays in $\COP$ to parametrize the factorizations of elements in $\partial\CP \cap \DN^\circ$. After that we manipulate these parametrizations to obtain information about the algebraic boundary of $\CP$, i.e., about the Zariski closure of $\partial\CP \cap \DN^\circ$ and its irreducible components. In other words, we reduce the problem to a computation of images of certain varieties under algebraic maps. This allows us to compute polynomials defining some of the irreducible components of the algebraic boundary and to compute the degree of the other components. We further discuss the uniqueness of completely positive factorizations in the boundary, rational factorizations of matrices $\CP$ and the cp-rank partition of $\CP$.

\subsection{The boundary of $\COP$}
Hildebrand classifies all extreme rays of $\COP$. The theorem goes as follows: 

\begin{thm}[\cite{Hildebrand}]\label{thm:char} Every extreme ray of $\COP$ is generated by a symmetric matrix $M$ of one the following four types: 
\begin{enumerate}
\item $M = vv^{\mathsf T}$, where $v\in \mathbb{R}^5$ has positive and negative entries. 
\item $M = e_ie_j^{\mathsf T} + e_je_i^{\mathsf T}$, where $\{e_1, \dots, e_5\}$ is the standard basis of $\mathbb{R}^5$.
\item $M = PDHDP^{\mathsf T}$, where $H$ is the Horn matrix below, $D$ is a positive diagonal matrix and $P$ is a permutation matrix. 

\begin{equation*} 
H = \left(
\begin{matrix*}[r]
   1 & -1 &1& 1& -1 \\
    -1 & 1& -1 &1& 1\\
    1 & -1 & 1 & -1 & 1\\
    1   & 1& -1&1& -1\\
    -1&1&1&-1&1
\end{matrix*} \right)
\end{equation*}
\item $M = PDT(\Theta)DP^{\mathsf T}$, where $T(\Theta)$ is a matrix defined in terms of five parameters below, $D$ is a positive diagonal matrix and $P$ is a permutation matrix. Here
\[ T(\Theta) = \left(
\begin{matrix*}
   1 &-\cos(\theta_1) &\cos(\theta_1+\theta_2)& \cos(\theta_4+\theta_5)&-cos(\theta_5)  \\
    -\cos(\theta_1)  & 1&-\cos(\theta_2)  &\cos(\theta_2+\theta_3)& \cos(\theta_5+\theta_1)\\
    \cos(\theta_1+\theta_2) & -\cos(\theta_2)  & 1 & -\cos(\theta_3)  & \cos(\theta_3+\theta_4)\\
    \cos(\theta_4+\theta_5)   & \cos(\theta_2+\theta_3)& -\cos(\theta_3) &1& -\cos(\theta_4) \\
    -\cos(\theta_5) &\cos(\theta_5+\theta_1)&\cos(\theta_3+\theta_4)&-\cos(\theta_4) &1
\end{matrix*} \right),
\] with $\Theta = (\theta_1, \theta_2, \theta_3, \theta_4,\theta_5)$ a tuple of positive real numbers satisfying, $\sum_{i=1}^5 \theta_i < \pi$. 

\end{enumerate}
\end{thm}

This parametrization of the extreme rays in $\COP$ suggests an approach to understand the algebraic boundary of $\CP$. In fact, $\partial\CP$ is the set of matrices $A$ such that $\langle A, X\rangle \ge 0$ for all the extreme rays described, with the additional constraint that equality holds for at least one ray. The part of the boundary shared by $\CP$ and $\DN$ is well understood: It consists of low rank matrices and matrices with some zero entries and corresponds to the types (1) and (2) in the theorem above. We will focus mainly on the other part of the boundary, namely, the matrices in the interior of $\DN$ and the boundary of $\CP$.
Therefore, any such $A$ must be invertible, since it would otherwise be on the boundary of $\Ss_5$ and thus also on the boundary of $\DN$. Furthermore, it is necessary that the entries of $A$ are all strictly positive to avoid the boundary of $\NN_5$. All the matrices in this part of the boundary are orthogonal to matrices of the type (3), which we call the {\em Horn orbit}, or (4), which we call the {\em Hildebrand orbit}. We will first work out the orthogonality to the parts that ignore the permutation matrices. With that in mind, the sets of completely positive matrices in $\partial \CP \cap \DNN_5^\circ$ orthogonal to matrices in the Horn or Hildebrand orbit will be called the Horn and Hildebrand {\em locus} respectively. 

We rely on the following simple remark, exploited heavily in \cite{MaxfieldMinc1963} to bound the completely positive rank in $\CP$. Assume that $A\in \partial \CP$ is orthogonal to a matrix $M$ as in parts (3) or (4) in Theorem \ref{thm:char}. Let $B$ be a nonnegative factorization of $A$, i.e., $A= BB^{\mathsf T}$. If $v_1, \dots v_k$ are the columns of $B$, then $v_i^{\mathsf T}Mv_i = 0$, i.e., each column of $B$ is a zero of the quadratic form associated to $M$. Since $M$ is copositve, this is equivalent to saying that every column is a global minimum of the quadratic form in the positive orthant $\mathbb{R}_{\ge 0}^5$.

We begin by mentioning a relevant theorem. 
\begin{thm}[\cite{Shaked-Monderer2013} Section 4]\label{thm:columnNumber} Assume $A\in \partial\CP$ is orthogonal to a matrix in the Horn orbit or a Hildebrand orbit. Then $A= BB^{\mathsf T}$ for a nonnegative square matrix $B \in \mathbb R_\geq^{5 \times 5}$.
\end{thm}

\subsection{The dual of the orbit of the Horn matrix}
The following theorem is hidden in a proof in \cite{Shaked-Monderer2013}:
\begin{thm}[\cite{Shaked-Monderer2013} Theorem 4.1] \label{thm:uniqueness} Let $H$ be the Horn matrix. A vector $v\in \RF$ is a solution to the equation $v^{\mathsf T}Hv=0$ if and only if it is in the union of cones 
\[\bigcup_{i=1}^5 \cone(e_i+e_{i+1}, e_{i+1}+e_{i+2}),\]
where the indices are taken modulo 5. Consequently, every matrix in $\CP\cap \DN^\circ$ is an invertible matrix $A$ such that the columns of any nonnegative matrix $B$ for which $A=BB^{\mathsf T}$ are (linearly independent) elements of the cone. \end{thm}

Notice that this restricts significantly the possible factor matrices $B$: the columns are a choice of five vectors in the union of such cones. Notice furthermore that if two of the columns are in the same cone, then the product $BB^{\mathsf T}$ has at least one entry equal to zero and is consequently in $\partial\DN$. It follows that in the Horn part of the boundary, the nonnegative factorizations of the matrices must contain exactly one vector in each cone. Notice that if $A = BB^{\mathsf T}$ and $\tilde B$ is obtained by permuting the columns of $B$, then $A = \tilde B \tilde B^{\mathsf T}$. Hence, the structure of any factorization is as follows:

\begin{lemma}\label{lem:factHorn} Any matrix $A$ orthogonal to $H$ and in $\partial\CP\cap \DN^\circ$ has a cp-factorization $A=BB^{\mathsf T}$, where $B$ is of the form: 
\begin{equation}\label{eqn:HornParam}B = \left(
\begin{matrix*}
   1& 0&0& y_4& y_5+1 \\
    y_1+1 & 1& 0 &0& y_5\\
    y_1 & y_2+1 & 1 & 0& 0\\
    0   & y_2& y_3+1&1& 0\\
    0&0&y_3&y_4+1&1
\end{matrix*} \right) \left(\begin{matrix*}
   z_1& 0&0& 0& 0\\
    0 & z_2& 0 &0& 0\\
    0& 0 & z_3& 0& 0\\
    0   & 0& 0&z_4& 0\\
    0&0&0&0&z_5
\end{matrix*} \right).
 \end{equation}
Here $y_1, y_2,y_3,y_4,y_5,z_1,z_2,z_3,z_4,z_5$ are positive real numbers. 
\end{lemma}
 The set of all matrices in the Lemma is a 10-dimensional relatively open cone in the space of $5\times 5$ matrices. The left action of the diagonal matrices on the factorization increase the dimension to create a hypersurface in a 15-dimensional space. The variety $V_{Horn}$ of the matrices orthogonal to $DHD$ for some diagonal matrix $D$ therefore is a hypersurface, which can be parametrized explicitly by modifying the lemma above. In order to achieve this we define some relevant varieties. 
\begin{defn} Let $W \subseteq \mathbb R_{\geq 0}^{5 \times 5}$ be the linear subspace of matrices of the form \begin{equation*}\label{eqn:bigspace}\left(
\begin{matrix*}
   y_{11}& 0&0& y_{41}& y_{51} \\
    y_{12} & y_{22}& 0 &0& y_{52}\\
    y_{13 }& y_{23}& y_{33} & 0& 0\\
    0   & y_{24}& y_{34}&y_{44}& 0\\
    0&0&y_{35}&y_{45}& y_{55}
\end{matrix*} \right).\end{equation*}
Let $Z_{Horn}\subset W$  be the hypersurface of all matrices of the form $DB$, where $B$ is a matrix as in \eqref{eqn:HornParam} and $D$ is a diagonal matrix with positive entries. 
\end{defn}
\begin{thm}\label{thm:implicit}
The Horn locus $V_{Horn}\subseteq \CP\cap \DN^\circ$ of matrices orthogonal to the matrices of the form $DHD$ is the image of $Z_{Horn}$ under the map $\varphi: W \to \Ss_n$ given by $\varphi(X)=XX^{\mathsf T}$. 
\end{thm}

\begin{thm}
The variety $Z_{Horn}$ is the vanishing locus of the polynomial \[s(x) = \det(H\circ X) \]
Here $\circ$ denotes the Hadamard product of matrices. 
\end{thm}
\begin{proof} There are two ways of verifying the above theorem. First, one may use the parame-trization from Theorem \ref{thm:implicit} to verify that $Z_{Horn}$ is contained in the vanishing locus $s(x)$. This computation is done by a computer algebra system. Furthermore, since $s(x)$ is irreducible (because determinants are) its vanishing locus is a hypersurface containing $Z_{Horn}$ and hence equal to its Zariski closure.

For a more conceptual approach, let $K$ be the convex cone generated by the matrices of the form $DHD$ where $D$ is a diagonal matrix. Notice that the extreme rays of $K$ are all extreme rays of $\COPP_5$, thus $K\subseteq \COPP_5$. It follows that the dual $K^*$ contains $\CP$ and shares a part of the boundary. There is a unique linear automorphism $f$ of $Sym_5$ such that $f(DHD)= D\mathbf{1}D$. Thus the cone $K$ is linearly isomorphic to the convex cone generated by $D\mathbf 1 D$, i.e., the cone whose extreme rays are rank one matrices.  This cone is $Sym_5$ which is self dual and its algebraic boundary is known to be given by  matrices of low rank, i.e., solutions to the equation $\det(X) =\det(\mathbf 1 \circ X) = 0$. The linear change of variables yields that $K^*$ has a boundary described by \eqref{eqn:HornParam}. The shared part of the boundary of $K^*$ and $\CP$ satisfies then the desired equation.   \end{proof}

\subsection{The dual of the orbit of Hildebrand matrices}
The following is implicit in the work of Hildebrand \cite[Section 3.2.2]{Hildebrand}:
\begin{thm}\label{thm:HildFact} Let $T(\Theta)$ be the five parameter matrix defined above. A vector $v$ is a solution to $v^{\mathsf T}T(\Theta) v=0$ if and only if it is a positive multiple of a column of the matrix
 \begin{equation*}\label{eqn:Spaces} S(\Theta):= \left(
\begin{matrix*}
   \sin(\theta_5)& 0&0& \sin(\theta_2)& \sin(\theta_3+\theta_4) \\
    \sin(\theta_4+\theta_5)& \sin(\theta_1)& 0 &0& \sin(\theta_3)\\
    \sin(\theta_4)& \sin(\theta_5+\theta_1)& \sin(\theta_2)& 0& 0\\
    0   & \sin(\theta_5)& \sin(\theta_1+\theta_2)&\sin(\theta_3)& 0\\
    0&0&\sin(\theta_1)&\sin(\theta_3+\theta_2)& \sin(\theta_4)
\end{matrix*} \right).
\end{equation*}

\end{thm}

As a consequence, we can parametrize the hypersurface dual to all the matrices of the form $DT(\Theta)D$ by matrices of the form $D_1 S(\Theta)D_2$ where $D_1$ and $D_2$ are diagonal matrices. In fact, the actions of $D_1$ and $D_2$ scale the diagonal products proportionally. Thus, with $W$ as in the previous section, we have the following parametrization:

\begin{thm}\label{thm:hildpoly} Let $V_{Hi}\subseteq \CP\cap\DN^\circ$ be the hypersurface of matrices orthogonal to some element in the (torus) orbit of Hildebrand matrices. Then $V_{Hi}$ is the image of the variety $Z_{Hi}$ associated to the ideal $\langle y_{11}y_{22}y_{33}y_{44}y_{55} - y_{13}y_{24}y_{35}y_{41}y_{52}\rangle$ in the coordinate ring of $W$, under the map $\varphi$ from Theorem \ref{thm:implicit}.
\end{thm}
\begin{proof} The variety $\tilde Z_{Hi}\subseteq W$ defined by matrices of the form $D_1S(\Theta)D_2$ is a hypersurface and every element of $D_1S(\Theta)D_2$ vanishes at the polynomial  $ y_{11}y_{22}y_{33}y_{44}y_{55} - y_{13}y_{24}y_{35}y_{41}y_{52}$. Then $\tilde Z_{Hi} \subseteq Z_{Hi}$ and they are both hypersurfaces, so they must coincide. 
\end{proof}
The theorem above yields a parametrization of an algebraic component of the desired boundary. The variety $V_{Hi}$ is therefore the implicitization of $Z_{Hi}$ under the map $\varphi$. However, the standard techniques using Gr\"obner bases yield no answer. This is for a good reason, since the numerical implicitization algorithm \cite{ChenKileel} implemented in the HomotopyContinuation package in Julia  \cite{Julia}, yields the following:
\begin{cor}\label{cor:degree}  The degree of $V_{Hi}$ is $320$.
\end{cor}

This computation can be certified using interval arithmetics \cite{Breiding2021}, which means that with very high probability, the given degree is correct but at least that this number is a robust lower bound. The degree being equal to 320 means that the expected number of monomials in the polynomial defining $V_{Hi}$ is close to $\binom{334}{14}\cong 10^{24}$. In particular, the  polynomial defining $V_{Hi}$ is likely impossible to write down. Of course, we could be optimistic and hope that the defining polynomial be very sparse and tractable, but chances of this seem very small.  

\subsection{Uniqueness of the factorization} 
Our experiments below need the uniqueness of factorizations. The following results are classical, but we briefly summarize them for the sake of completeness. Notice that if $A=BB^{\mathsf T}$ and $P$ is a permutation matrix, then $A=(BP)(BP)^{\mathsf T}$, thus permuting columns of $B$ maintains the factorization. Up to this, we will see that factorizations of elements in $\partial\CP \cap\DN^\circ$ are unique. 

The uniqueness comes from an old lemma that is due to Baumert \cite[Lemma 4.5]{Baumert67}:

\begin{lemma}
Let $M\in \COP$ be a full rank extreme ray. There are exactly 5 nonnegative solutions of the quadratic form $v^{\mathsf T}Mv = 0$ up to scaling, i.e., they come from 5 different lines in $\mathbb{R}^5$.
\end{lemma}
The original lemma also discusses the zero patterns of the solutions, but a combination of this lemma with the fact that matrices $\partial\CP \cap\DN^\circ$ are invertible tells us that any factorization of an element normal to $A$ in $\partial\CP \cap\DN^\circ$ corresponds to choosing one element of each of the lines. 
Since the lines must be linearly independent (the factorization matrix is invertible), different choices of points on the lines for the factorization of an orthogonal matrix yield different matrices. As a matter of fact, the columns of the matrices in Lemma~\ref{lem:factHorn} and Theorem~\ref{thm:HildFact} are generators of the explicit solutions of the corresponding quadratic forms in $\COP$. 

\subsection{Rational points on $\partial\CP \cap\DN^\circ$}

We now observe that the matrices in $\partial\CP \cap\DN^\circ$ with rational entries are dense in that part of the boundary. This is a consequence of the theorems in the previous two sections and in particular this allows for the exact computation of rational matrices in $\partial\CP \cap\DN^\circ$. With this in mind, let $\Sym(\mathbb{Q})$ be the set of symmetric matrices with rational entries. We summarize the result as follows:

\begin{thm}\label{thm:rationalPoints} 
 The set of rational matrices in $\partial\CP \cap\DN^\circ$ is dense (with the Euclidean topology), that is, $\partial\CP \cap\DN^\circ \subseteq \overline{\partial\CP \cap\DN^\circ \cap \Sym(\mathbb{Q})}$. 
\end{thm}

\begin{proof}
The Horn orbit (and its conjugates with permutation matrices) is defined by a determinant equation that is linear in each entry on the diagonal. If $A$ is in the Horn orbit, we can approximate each entry different from $A_{11}$ with an arbitrarily close rational number.  To force the matrix to be in the Horn orbit, the first entry is then a uniquely determined rational number.

For a matrix $A$ in the the Hildebrand orbit let $B$ be a nonnegative matrix with $A = BB^{\mathsf T}$. To construct a rational approximation $\tilde B$, approximate every entry except $\tilde B_{11}$ by a rational number and then set $\tilde B_{11}$ to satisfy the equation defining $Z_{Hi}$. Then $\tilde A =\tilde B \tilde B^{\mathsf T}$ is a good approximation of $A$ and also in the boundary. 
\end{proof}

\subsection{The action of the permutation matrices}
Notice that in parts (3) and (4) of Theorem \ref{thm:char}, there are some permutation matrices $P$ that we have ignored so far. To understand all the components of the boundary of $\CP$ contained in $\DN^\circ$, we have to consider the effect of these matrices, which permute rows and endow algebraic components of the boundary with an action of the symmetric group $\mathfrak{S}_5$. 

We need to find the stabilizer of the action on the Horn and Hildebrand parts of the boundary. For this, we recall that if $A\in \CP\cap \DN^\circ$ and $A=BB^{\mathsf T}$ with $A$ nonnegative, then any matrix $Q \in \mathcal{O}(5)$ yields a factorization $A=(BQ)(BQ)^{\mathsf T}$. In general, the matrices of $\mathcal{O}(5)$ have negative entries, so preserving positivity is nontrivial and known to fail in $\partial\CP \cap \DN^\circ$ unless $Q$ is a permutation matrix. In this case $Q$ permutes columns of $B$ and preserves positivity. 
Let two generic points in $\partial\CP \cap \DN^\circ$ be either in the Horn or in the Hildebrand locus. Then they are in the same algebraic component if they share the same zero pattern up to this column permutation that preserves the factorization. This is because generic elements in  $\partial\CP \cap \DN^\circ$ are orthogonal to a unique ray of type (3) or (4) from Theorem \ref{thm:char}. 

It follows that the stabilizer of the algebraic components of  $\partial\CP \cap \DN^\circ$ under the $\mathfrak{S}_5$-action is a dihedral group that shifts rows cyclically or reflects them. As a consequence, the orbit of each algebraic component consists of 12 varieties. Putting this together we obtain the following:

\begin{thm}
The Zariski closure of the hypersurface $\partial\CP \cap \DN^\circ$ is the vanishing locus of a 15 variable polynomial of degree $3900 =(320 + 5)12$.
\end{thm}

\subsection{Completely positive rank and the interior of $\CP$}\label{sec:intcp}
Another rather interesting aspect of the cone $\CPP_n$ comes with regard to the so called cp-rank. This will actually help us understand the cone $\CP$ in more detail. Throughout this section, the topology on $\CP$ and all of its subsets is the Euclidean topology.

\begin{defn} The cp-rank $cp(A)$ of a matrix $A \in \CPP_n$ is the minimal size $k$ such that there is a nonnegative $n\times k$ matrix $B$ with $A = BB^{\mathsf T}$. The cp$^+$-rank $cp^+(A)$ of $A$ is the smallest such $k$ for which $B$ can be taken to have strictly positive entries. 
\end{defn}

As a corollary of Thm.~\ref{thm:columnNumber}, we know that all matrices in $\partial \CP$ can have cp-rank at most 5. We would like to understand the cp-ranks and cp$^+$-ranks in the interior of $\CP$. It is known \cite[Theorem 5.1]{Bomze2015} that they agree generically on an open subset of the interior of $\CP$. For that, we consider a few extra subsets of $\CPP_n$.

\begin{defn}  Let $\CR$ be the set of all matrices $A\in \CPP_n$ such that $cp(A) \le k$. 
\end{defn}

\begin{thm}\label{closedness} For every $n,k$ the set $\CR$ is a closed (not necessarily convex) cone.
\end{thm}
\begin{proof} Let $\mathbb R_{\geq}^{n \times k}$ be the set of all $n\times k$ matrices with nonnegative entries and let $\varphi: \mathbb R_{\geq}^{n \times k} \to \CPP_n$ be given by $\varphi(X)= XX^{\mathsf T}$. This is a continuous map whose image is precisely $\CR$.  Now let $A$ be a matrix in the closure of $\CR$ and choose a sequence of matrices $\{
A_i\}_{i=1}^\infty$ that converges to $A$. For each $i$, pick a matrix $B_i$ in the preimage of $A_i$ under $\varphi$. Notice that $\{B_i\}_{i=1}^\infty$ is a sequence of matrices that are entrywise bounded: indeed, $(B_i)_{j,\ell} \le \sqrt{\sum_m (B_i)_{j,m}^2} =\sqrt{(A_i)_{j,j}}$ and the diagonal entries of the matrices $A_i$ are bounded since they converge to $A$. Thus $\{B_i\}_{i=1}^\infty$ has a convergent subsequence in the Euclidean topology and this subsequence converges to a matrix $B$ such that $\varphi(B)= A$. 
\end{proof}

Our goal is to understand $\partial \CRF \cap \CP^\circ$. Since the matrices in this set are in the interior of $\CP$, they are all invertible and by the theorem above their cp-rank is exactly 5. 

\begin{thm}\label{thm:Interior} The intersection $\partial \CRF \cap \CP^\circ$ consists of the collection of all matrices whose cp-rank is equal to 5 and whose cp$^+$-rank is equal to 6. 
\end{thm}
\begin{proof}
The interior of $\CP$ consists of invertible matrices, meaning that every matrix in the interior has cp-rank at least 5. Combining \cite[Theorem 4.1]{Bomze2015} together with the fact that maximal cp-rank in $\CP$ is equal to 6 tells us that 
\begin{equation*}
\max\{cp^+(M)\, | \, X\in \CP^\circ\}=\max\{cp(M)\,| \, X\in \CP^\circ\} =6.
\end{equation*}

By \cite[Corollary 2.7]{Bomze2015} and Theorem \ref{closedness}, all the points of $\overline \CRF \backslash (\CRF^\circ \cup \partial\CP)$ are matrices whose cp-rank is 5 and whose cp$^+$ rank is 6. According to Thm.~\ref{closedness}, every point in $\overline \CRF \backslash (\CRF^\circ \cup \CP)$ is an accumulation point and hence an element of $\partial\CRF \cap \CP^\circ$. 
\end{proof}

We are interested in the factorization of elements in $\partial \CRF \cap \CP^\circ$. Theorem \ref{thm:Interior} implies that the interior of the cone is contained in the closure of two disjoint open sets, whose boundary is described in terms of matrices whose factorizations are quite special. In particular, all of them must have a large number of entries that are zero.

While figuring out the exact pattern is complicated, there are still a few things that could be said. By manipulating the zeros as in \cite{Dickinson2010}, we can assume that every column of a $5\times 5$ factorization contains at least one zero. Furthermore, in a factorization with the smallest possible number of zeros, there can be no pair of columns such that the zero entries of one column are also zero entries of the other column, i.e., the zero pattern of one column cannot be contained in the zero pattern of another.

Since the dimension of the variety $\partial \CRF \cap \CP^\circ$ is 14, the number of zeros in the factorization is between 5 and 11. The possible patterns can be classified combinatorially and are many. But the main issue is that even if we have such potential patterns, there is little we know to find parametrizations with the given pattern. The difficulty is perhaps the lack of convexity of the relevant cones making it impossible to use duality techniques as for the boundary of $\CP$. Furthermore, even though the factorizations should in principle be unique (in the spirit of Theorem \ref{thm:uniqueness}), there could be other factorizations that are not nonnegative but very close. This is an issue even in the boundary $\partial \CP$, as will be discussed in Sec.~\ref{ssc:condition}, and it makes the generation of possible examples very difficult. And lastly, it is difficult to verify that possible polynomial equations yield a parametrization, because this will be further restricted by an unknown number of inequalities that cannot be predicted using the dimension of the set.

Nevertheless, we believe that we can produce exact examples of such matrices, as will be shown in Sec.~\ref{ssc:numinner}. The numerical experiments suggest that the number of zeros in the square factorization would be 10 or 11. We remark that having factorizaions with 11 zeros seems implausible and worth further consideration. On the one hand, since $\partial \CRF \cap \CP^\circ$ is a hypersurface, its dimension is 14 and the matrices with 11 zeros would provide a parametrization of an algebraic component of a $\partial \CRF \cap \CP^\circ$. On the other hand, this boundary remains elusive: Simply inserting random values into a given zero pattern will not produce a desired factor matrices, because of the unknown polynomial inequalities. We discuss this in detail in Sec.~\ref{ssc:numinner}.

\section{A numerical algorithm}
\label{sec:numalg}

A number of algorithms for the computation of the completely positive factorization of a matrix $A$ have been proposed. Since computing the cp-rank or even deciding whether a matrix allows such a factorization is a co-NP-complete problem \cite{MurtyKabadi87,Sponsel2012}, these algorithms must assume that the cp-rank $r$ is known and preset. If this choice turns out to be wrong, the algorithm will either not terminate or it will present a solution that does not fulfill all criteria, i.e., it is either not symmetric, not nonnegative, or it does not yield the desired matrix.

We can distinguish two general classes of algorithms. The first approach is to begin with any symmetric factorization $A = BB^{\mathsf T}$ and then to iteratively alter $B \in \mathbb R^{n \times r}$ such that it becomes nonnegative. This can for example be done by picking an initial orthogonal matrix $Q \in \mathcal O(r)$ and projecting it onto the polyhedral cone
\begin{equation*}
\mathcal P = \{ R \in \mathbb R^{r \times r} : BR \geq 0 \}.
\end{equation*}
After this, the result will in turn be projected back onto $\mathcal O(r)$, whereupon the procedure repeats until $BQ$ remains nonnegative and therefore constitutes a solution. This method has been proposed in \cite{Groetzner2020}, it is very fast and it returns an exact solution since $BQ\bigl(BQ\bigr)^{\mathsf T} = BB^{\mathsf T} = A$, at least up to numerical accuracy for the representation of $Q$. However, it has the severe drawback that the set of orthogonal matrices $Q$ with $BQ \geq 0$ needs to be sufficiently large, or else the algorithm often fails to converge. In the boundary $\CPP_n$, the factor matrices have many zeros and therefore this set is of high codimension in $\mathcal O(r)$. Thus, if $A$ is (very) close to the boundary $\partial \CPP_n$, the algorithm will fail and it cannot be used to distinguish these matrices from the ones that do not allow a completely positive factorization of rank $r$.

The alternative approach consists of algorithms that approximate the exact factorization while maintaining symmetry and nonnegativity, see for example \cite{Ding2005}. In our experiments, we used a version of these methods that to our knowledge has not been applied to the problem at hand, although it uses only standard tools of numerical approximation. The goal is to minimize the function
\begin{equation*}
f(B) = \bigl\| A - B B^{\mathsf T} \bigr\|^2.
\end{equation*}
In order to guarantee nonnegativity of $B$, we can write its entries as squares, resulting in $B = C \circ C$. Together with a factor that simplifies the gradient, we aim to minimize the smooth function
\begin{equation*}
g(C) = \frac18 \Bigl\| A - \bigl( C \circ C \bigr) \bigl( C \circ C \bigr)^{\mathsf T} \Bigr\|^2.
\end{equation*}

This is done using standard tools from numerical optimization. Since we deal with matrices, the MATLAB-toolbox {\tt manopt} allows for an easy implementation and good performance \cite{manopt}. We applied the provided trust region scheme, because it proves to be faster than equally applicable methods like gradient descent or nonlinear conjugate gradients.\footnote{The trust region method consists of solving a quadratic approximation of the cost function $g$ on a small {\em trusted} region around the current iterate. Its size is adapted throughout the procedure according to prior performance. In {\tt manopt}, the Hessian is approximated using finite differences of the gradient. We do not go into more detail here and simply set the next iterate as $C_{k+1} := {\rm TrustRegion}\bigl(C_k,g,\nabla g(C_k)\bigr)$. See \cite{Nocedal2006} for an accessible introduction into these methods.} The only other ingredient that we need is the Euclidean gradient of our function $g$, which can be readily given as
\begin{equation*}
\nabla g(C) = \Bigl( \bigl(C \circ C\bigr) \bigl(C \circ C \bigr)^{\mathsf T} \bigl(C \circ C \bigr) \Bigr) \circ C - \Bigl( A \, \bigl( C \circ C \bigr) \Bigr) \circ C.
\end{equation*}
See Algorithm~\ref{alg:method} for the implementation in pseudo-code.

\begin{algorithm}[t]
\label{alg:method}
\caption{Trust region scheme for the computation of the cp-factorization. }
\SetKwInput{Input}{Input~}
\SetKwInOut{Output}{Output~}
\Input{$A \in \DNN_n$, initial point $C_0 \in \mathbb R_\geq^{n\times r}$, tolerance $\varepsilon > 0$, maximal number of iterations ${\rm maxit}$}
$k = 0$\; 
\While{$\| \nabla g(C_k) \| > \varepsilon$ {\rm \textbf{and}} $k < {\rm maxit}$}
{
Perform trust region step $C_{k+1} := {\rm TrustRegion}\bigl(C_k,g,\nabla g(C_k)\bigr)$\;
$k \leftarrow k + 1$\;
}
$B = C_k \circ C_k$\;
\Output{Approximate factor matrix $B \in \mathbb R_\geq^{n\times r}$}
\end{algorithm}

\subsection{Enforcing zeros}

We know that the concept of zero patterns in the factorization plays an important role in the characterization of the boundary $\partial \CP$ or $\CRF \cap \CP^\circ$. For the purpose of numerical experiments, it can therefore be beneficial to enforce a specific zero pattern in the solution, for example in order to find a matrix in these sets. Hence, we note that if the initial point of the optimization, say $C_0$, has a given zero pattern, then so does the gradient $\nabla g(C_0)$. Any next iterate $C_1$ that is a result of a gradient related optimization step will therefore still have the given zero pattern, and so on. This means that enforcing a zero pattern in the algorithm can be done by simply starting out with this pattern. However, we also need to take into consideration that we effectively look for a solution on the intersection of a linear space with a hypersurface (e.g., a part of the boundary), which can lead to undesired effects.

\subsection{Condition of the reconstruction problem}
\label{ssc:condition}

The above approach essentially means that we are trying to find a global minimum of a polynomial of order 8, resulting in many local minima and possibly in an ill-conditioned problem. However, we can simply restart the method if it does not produce an actual solution (meaning $g(C) = 0$). This works as long as we know the cp-rank of our matrix. Otherwise, the algorithm will produce an approximation of our matrix and we can use several tries to find the best one. This seems to work well in practice.

The issue of conditioning is more problematic. Na\"ively, if $g(C) = \varepsilon$ for a small error $\varepsilon > 0$, one would expect the error $\| B - \bigl(C \circ C\bigr) \|$ to be of order $O(\varepsilon^{\frac18})$, where $B$ is an exact factorization of $A$. However, we can find examples that illustrate that the problem is more severe: 
We begin by constructing a matrix in the Horn part of the boundary (see~\eqref{eqn:HornParam} and Sec.~\ref{ssc:constbound}):

\begin{equation*}
B = \begin{pmatrix} 1 & 0 & 0 & 4 & 6 \\
2 & 1 & 0 & 0 & 5 \\
1 & 3 & 1 & 0 & 0 \\
0 & 2 & 4 & 1 & 0 \\
0 & 0 & 3 & 5 & 1
\end{pmatrix} \quad \Rightarrow \quad A = B B^{\mathsf T} = \begin{pmatrix} 53 & 32 & 1 & 4 & 26 \\
32 & 30 & 5 & 2 & 5 \\
1 & 5 & 11 & 10 & 3 \\
4 & 2 & 10 & 21 & 17 \\
26 & 5 & 3 & 17 & 35
\end{pmatrix}
\end{equation*}
Our algorithm is often able to reconstruct the matrix $B$ but in other cases (depending on the random starting point), a different factor matrix is produced: 

\begin{equation*}
\tilde B = \begin{pmatrix} 
4.9294\ldots & 0 & 0 & 1.5165\ldots & 5.1382\ldots \\
5.1323\ldots & 1.3996\ldots & 0 & 0 & 1.3041\ldots \\
0.2029\ldots & 2.8286\ldots & 1.7199\ldots & 0 & 0 \\
0 & 1.4290\ldots & 3.4642\ldots & 2.6377\ldots & 0 \\
0 & 0 & 1.7443\ldots & 4.1542\ldots & 3.8341\ldots
\end{pmatrix}
\end{equation*}
The resulting matrix $\tilde A = \tilde B \tilde B^{\mathsf T}$ is numerically almost indistinguishable from $A$:
\begin{equation*}
\| A - \tilde A \| \approx 2.143633117337326 \cdot 10^{-11}.
\end{equation*}
We can even find an orthogonal transformation $Q \in \mathcal O(5)$, for which, up to numerical accuracy, it holds $B Q = \tilde B$:
\begin{equation*}
Q = \begin{pmatrix} 
0.2262\ldots & -0.0611\ldots & 0.6472\ldots & -0.6666\ldots & 0.2860\ldots \\
-0.0333\ldots & 0.9922\ldots & 0.1194\ldots & 0.0087\ldots & -0.0116\ldots \\
0.0764\ldots & -0.0931\ldots & 0.7246\ldots & 0.6318\ldots & -0.2474\ldots \\
-0.2352\ldots & -0.0044\ldots & 0.0346\ldots & 0.3439\ldots & 0.9084\ldots \\
0.9416\ldots & 0.0562\ldots & -0.2014\ldots & 0.1951\ldots & 0.1779\ldots 
\end{pmatrix}.
\end{equation*}
Similar examples can be constructed for the Hildebrand part of the boundary.

\begin{rem}
Note that the factor matrix $\frac12 B + \frac12 \tilde B$ (or any other linear combination) does not yield a matrix close to $A$, even though they are in the same linear subspace described by~\eqref{eqn:HornParam}. This is because the mixed terms $B \tilde B^\mathsf{T}$ and $\tilde B B^\mathsf{T}$ yield very unpredictable results.

Furthermore, we have tested transformations $B \tilde Q$ for some orthogonal matrices on the geodesic in $\mathcal O(5)$ between the identity and $Q$. As one might expect, all of these transformations result in some negative entries, suggesting that $Q$ is the one transformation matrix that makes $BQ$ nonnegative, while at the same time being orthogonal up to machine precision.
\end{rem}

Since the factorization of a matrix in the boundary is unique, this is not actually a problem of ill-conditioning, because we cannot find arbitrarily close approximations of $A$. The approximation problem is only {\em numerically} ill-posed, which nevertheless makes it much more difficult to solve.

Ultimately, the problem of numerical ill-posedness is inherent in the problem structure and not in the algorithm itself. Any algorithm that is subject to numerical noise will suffer from it. In fact, we can learn many interesting things from this: Even for matrices with a unique cp-factorization (i.e., those in $\partial \CP$) there exist orthogonal transformations that will preserve the factorization up to numerical accuracy. 

\subsection{Finding unknown factorizations}
With the above caveat in mind, we can nevertheless use our algorithm to find factorizations of difficult matrices. The article \cite{Bomze2014} generates a number of matrices of different sizes that are known to be cp-decomposable with known (high) cp-rank, but their factorization is unknown. Our algorithm was able to find the (approximate) cp-factorization $\tilde B$ of the $7 \times 7$ matrix

\begin{equation*}
A = \begin{pmatrix}
163 & 108 & 27 & 4 & 4 & 27 & 108 \\
108 & 163 & 108 & 27 & 4 & 4 & 27 \\
27 & 108 & 163 & 108 & 27 & 4 & 4 \\
4 & 27 & 108 & 163 & 108 & 27 & 4 \\
4 & 4 & 27 & 108 & 163 & 108 & 27 \\
27 & 4 & 4 & 27 & 108 & 163 & 108 \\
108 & 27 & 4 & 4 & 27 & 108 & 163 \\
\end{pmatrix}
\end{equation*}
that has known cp-rank 14, up to an accuracy $\| A - \tilde B \tilde B^{\mathsf T} \| \approx 1.179332674168671 \cdot 10^{-8}$. The peculiar zero pattern together with the fact that many of the entries seemed to be repeated led us to deduce that the exact factorization is
\begin{equation*}
\setcounter{MaxMatrixCols}{14}
B = \begin{pmatrix}
\sqrt{27} & \frac{51}{\sqrt{27}} & \sqrt{27} & 0 & 0 & 0 & 0 & \sqrt{\frac23} & 0 & 0 & \sqrt{6} & \sqrt{6} & 0 & 0 \\
0 & \sqrt{27} & \frac{51}{\sqrt{27}} & \sqrt{27} & 0 & 0 & 0 & 0 & \sqrt{\frac23} & 0 & 0 & \sqrt{6} & \sqrt{6} & 0 \\
0 & 0 & \sqrt{27} & \frac{51}{\sqrt{27}} & \sqrt{27} & 0 & 0 & 0 & 0 & \sqrt{\frac23} & 0 & 0 & \sqrt{6} & \sqrt{6} \\
0 & 0 & 0 & \sqrt{27} & \frac{51}{\sqrt{27}} & \sqrt{27} & 0 & \sqrt{6} & 0 & 0 & \sqrt{\frac23} & 0 & 0 & \sqrt{6} \\
0 & 0 & 0 & 0 & \sqrt{27} & \frac{51}{\sqrt{27}} & \sqrt{27} & \sqrt{6} & \sqrt{6} & 0 & 0 & \sqrt{\frac23} & 0 & 0 \\
\sqrt{27} & 0 & 0 & 0 & 0 & \sqrt{27} & \frac{51}{\sqrt{27}} & 0 & \sqrt{6} & \sqrt{6} & 0 & 0 & \sqrt{\frac23} & 0 \\
\frac{51}{\sqrt{27}} & \sqrt{27} & 0 & 0 & 0 & 0 & \sqrt{27} & 0 & 0 & \sqrt{6} & \sqrt{6} & 0 & 0 & \sqrt{\frac23}
\end{pmatrix}.
\end{equation*}
We remark that our algorithm did not produce similar results for the other matrices in \cite{Bomze2014}, even over many tries, as it seems to run into local minima.

\section{Constructing examples}
\label{sec:consexp}

Before we use our algorithm to compute the cp-factorization of matrices in the different parts of $\CP$ that we have discussed above, we present some ideas of how to produce more or less generic matrices in these parts. First of all, let it be stated that picking a matrix $M \in \DN$, say of fixed Frobenius norm $\|M\| = 1$, uniformly at random is not entirely trivial. If we just pick 15 nonnegative entries of $M$ on the upper triangle and normalize, we will most likely not have a positive semidefinite matrix. This could be remedied with the so-called rejection algorithm, where such a randomly chosen matrix $M \in \NN_5$ is rejected precisely when it has some negative eigenvalues. In any case, a generic matrix in $\DN$ will either have cp-rank 5 or 6, or it will not allow for a completely positive factorization. In the following, we will discuss several other strategies of how to obtain interesting matrices.

\subsection{Matrices in $\partial \CP$}\label{ssc:constbound}

The parts of the boundary $\partial \CP$ that are also in the boundary of $\DN$ are very easy to reproduce, either by picking a matrix of rank 4 or less (we call this part of the boundary $V_{rk4}$), or by keeping one or many entries equal to zero (for the boundary part $V_{zero}$). We have identified the remaining parts as $V_{Horn}$ and $V_{Hi}$. Since the algebraic and semialgebraic equations that describe these sets are rather complicated, we have to use an indirect approach and choose the factor matrices.

In order to obtain points in $V_{Horn} \subset \partial \CP$, one can simply choose 15 variables $x_1,\ldots,x_5$, $y_1,\ldots,y_5$, $z_1,\ldots,z_5$ (e.g.~uniformly at random) and produce a factor matrix 
\begin{equation}\label{eq:Hornmatrix}
\setlength\arraycolsep{2.5pt}
B_{Horn} = 
\left(\begin{matrix*}
x_1 & 0 & 0 & 0 & 0 \\
0 & x_2 & 0 & 0 & 0 \\
0 & 0 & x_3 & 0 & 0 \\
0 & 0 & 0 & x_4 & 0 \\
0 & 0 & 0 & 0 & x_5
\end{matrix*}\right)
\left(\begin{matrix*}
1 & 0 & 0 & y_4 & y_5 + 1 \\
y_1 + 1 & 1 & 0 & 0 & y_5 \\
y_1 & y_2 + 1 & 1 & 0 & 0 \\
0 & y_2 & y_3 + 1 & 1 & 0 \\
0 & 0 & y_3 & y_4 + 1 & 1
\end{matrix*}\right)
\left(\begin{matrix*}
z_1 & 0 & 0 & 0 & 0 \\
0 & z_2 & 0 & 0 & 0 \\
0 & 0 & z_3 & 0 & 0 \\
0 & 0 & 0 & z_4 & 0 \\
0 & 0 & 0 & 0 & z_5
\end{matrix*}\right).
\end{equation}
If all variables are chosen to be nonnegative, this will yield an element in the boundary $A_{Horn} = B_{Horn} B_{Horn}^{\mathsf T} \in V_{Horn} \cap \partial \CP$. 

In the case of $V_{Hi}$, this is done similarly by choosing $x_1,\ldots,x_5$, $\theta_1,\ldots,\theta_5$, $z_1,\ldots,z_5$ and replacing the middle matrix in~\eqref{eq:Hornmatrix} by $S(\Theta)$ in~\eqref{eqn:Spaces}, keeping in mind that $\sum_i \theta_i < \pi$ must hold. The resulting matrix $A_{Hi} = B_{Hi} B_{Hi}^{\mathsf T}$ will however be only approximately in the boundary $V_{Hi}$, since the trigonometric functions can only be approximated in general. In order to obtain matrices in $V_{Hi}$ with rational entries, we can round all entries of $B_{Hi}$ to a given number of digits except for one, which will then be determined by the fact that $y_{11}y_{22}y_{33}y_{44}y_{55} - y_{13}y_{24}y_{35}y_{41}y_{52} = 0$ must hold.

\subsection{Tangent spaces of the boundary and their orthogonal lines}
\label{ssc:consttang}

Examples of matrices in $\DN$ that do not allow for a completely positive factorization of any rank, i.e., of matrices in $\DN \setminus \CP$, are rare. Some are given in \cite{Berman2003}. In theory, all extreme rays of $\DN$ are known and those that are {\em not} extreme rays of $\CP$ should in principle produce more examples, but obtaining rational matrices this way could be difficult due to the involvement of trigonometric functions.

With the knowledge of the boundary $\partial \CP \cap \DN^\circ$, we are able to present a more systematic procedure to produce such examples for the case of $5 \times 5$ matrices. For this, we calculate a normal direction to this part of the boundary. Then, any point in this direction will be outside $\CP$. If we take a small enough step, we will often find cases that are still in $\DN$. We describe this procedure in more detail for the Hildebrand locus $V_{Hi}$:

Let $B_{Hi}$ be the factor matrix as in Sec.~\ref{ssc:constbound}. Theorem~\ref{thm:hildpoly} states that then $y_{52} = \frac{y_{11}y_{22}y_{33}y_{44}y_{55}}{y_{13}y_{24}y_{35}y_{41}}$ and therefore, we have a smooth parametrization $\varphi : \mathbb R^{14} \rightarrow V_{Hi} \subset \mathbb R^{15}$. Taking the orthogonal complement with respect to the Bombieri-norm in $\mathbb R^{15}$ of the Jacobian $\nabla \varphi(y)$ reveals the normal direction in $\mathbb R^{15}$.

The same procedure works for the Horn locus $V_{Horn}$, either by using the much more complicated polynomial equation $\det(H \circ X) = 0$, or by taking the derivative of the 15-dimensional parametrization~\eqref{eq:Hornmatrix} and using the fact that the Jacobian will be rank-deficient.

\subsection{Matrices in $\CP \setminus \CRF$}
\label{ssc:cp6}

It is also interesting to obtain matrices of cp-rank 6, i.e., matrices in $\CP \setminus \CRF$. In \cite{Shaked-Monderer2013}, a Kronecker-structured matrix of cp-rank 6 is given. However, it has rank 4 and it is therefore an element of $\partial \DN$. In order to obtain matrices in the interior, we can use some small perturbations of these matrices by matrices in the interior of $\CP$. Alternatively, a brute force sampling of $\DNN_5$, as performed in the next section, sometimes (albeit rarely) yields a matrix of cp-rank 6.

\section{Numerical experiments}

In the following, we perform a number of experiments that serve to highlight some aspects of the cp-cone for the special case $n = 5$. We will use our approximation algorithm and we will see that even for this small case, the cp-factorization problem is very complicated.

\subsection{Generic matrices in $\DNN_5$}

\begin{table}[t]
\centering
\begin{tabu}{ c | c | c | c}
\toprule
$A \in \DNN_5$ & $A \in \DNN_5 \setminus \CP$ & $A \in \CRF$ & $A \in \mathcal{CPR}_5(6)$ \\ 
\midrule
50000 & 8 & 49991 & 1 \\
\bottomrule
\end{tabu}
\vspace*{.2cm} 
\caption{Occurrences of matrices out of 50000 random draws in $\DNN_5$ with the rejection algorithm.}
\label{tab:generic}
\end{table} 

As discussed above, a generic matrix in $\DNN_5$ can be picked using the rejection algorithm. In Table~\ref{tab:generic}, we report on the occurrences of matrices in different full-dimensional parts of $\DNN_5$ out of 50000 of such picks. For each pick, we generated the 15 unique entries of the symmetric matrix $A$ uniformly at random in $(0,1)$. Then we calculated the eigenvalues of this matrix and rejected the matrix if any one of the eigenvalues was negative. Note that we did not normalize the matrix, since linear scaling does not have an effect on this experiment. The resulting matrix $A$ will be an element of $\DNN_5$. We then ran our algorithm 10 times with cp-rank 5 and we consider $A \in \CRF$ if the resulting factorization $B B^{\mathsf T}$ of any of these runs has an error
\begin{equation*}
\| A - B B^{\mathsf T} \| < 10^{-8}.
\end{equation*}
If out of these 10 runs none was successful, we increased the cp-rank to 6 and repeated the experiment. If again no run produced an error smaller than $10^{-8}$, we consider the matrix $A \in \DNN_5 \setminus \CP$.

We can see that the vast majority of generic matrices in $\DNN_5$ has cp-rank 5. A small number of matrices does not allow for a cp-factorization and only one matrix seems to have cp-rank 6.

\subsection{Experiments on the boundary}
In contrast to \cite{Groetzner2020}, our numerical algorithm allows us to factorize matrices in $\partial \CP$. The factorizations of these matrices are unique and should in principle be reconstructable. However, as we have discussed in Sec.~\ref{ssc:condition}, this is numerically not stable.

Table~\ref{tab:boundary} shows an experiment for the reconstructability of different parts of the cone. Since the boundary is of lower dimension, a random pick in $\DNN_5$ will never produce a matrix on the boundary. Therefore, we resort to the procedures presented in Sec.~\ref{sec:consexp} in order to obtain more or less random elements of $\partial \CP$.

With these methods, we randomly generated 100 matrices in the interior $\CP^\circ$, the Horn locus $V_{Horn}$, the Hildebrand locus $V_{Hi}$, the rank-deficient locus $V_{rk4}$, and those containing (at least) a zero in $V_{zero}$. For each of the matrices we performed our algorithm 10 times and considered it successful if the reconstructed matrix differs by less than $10^{-6}$ from the original matrix, in terms of Frobenius error. We can see that our algorithm successfully finds such a factorization in all cases.

Furthermore, for matrices in the interior, in the Horn locus, and in the Hildebrand locus, we can also compare the reconstructed factor matrices. This is not possible for rank-deficient matrices, because we applied our algorithm with initial rank 5 and the actual factor matrix has size $5 \times 4$. Similarly, we do not know the original factor matrix for matrices in $V_{zero}$, and therefore we cannot compare our results to it. We consider the factorization to be successful if the factor matrices differ by less than an error of $10^{-3}$ from the original factors. It is unsurprising that we never find the original factorization of a matrix in the interior of the cone since there are infinitely many factorizations. For the other two loci, we encounter the problem of numerical ill-posedness as discussed in Sec.~\ref{ssc:condition} in about three quarters of the cases, suggesting that finding a unique factorization is numerically very difficult.

\begin{table}[t]
\centering
\begin{tabu}{ >{\centering\arraybackslash}p{1.3cm} | >{\centering\arraybackslash}p{1.3cm} | >{\centering\arraybackslash}p{1.3cm} | >{\centering\arraybackslash}p{1.3cm} | >{\centering\arraybackslash}p{1.3cm} | >{\centering\arraybackslash}p{1.3cm} | >{\centering\arraybackslash}p{2.5cm} | >{\centering\arraybackslash}p{2.5cm}  }
\toprule
\multicolumn{2}{c|}{$A \in \CP^\circ$} & \multicolumn{2}{c|}{$A \in V_{Horn}$} & \multicolumn{2}{c|}{$A \in V_{Hi}$} & $A \in V_{rk4}$ & $A \in V_{zero}$ \\
a) & b) & a) & b) & a) & b) & & \\
\midrule
100 & 0 & 100 & 28 & 100 & 21 & 100 & 100 \\
\bottomrule
\end{tabu}
\vspace*{.2cm} 
\caption{Success rate of the algorithm out of 100 tries for the computation of any factorization (a) and for the original factorization (b) for different parts of $\CP$.}
\label{tab:boundary}
\end{table} 

\subsection{Approximation problems}

Using the method discussed in Sec.~\ref{ssc:consttang}, we can produce examples in $\DNN \setminus \CP$. Since the cone is also convex, we know the best approximations in $\CP$ for these matrices and we can use our algorithm to try to find them.

Table~\ref{tab:approx} shows the success rate for this approximation problem. For 100 random matrices in $V_{Hi}$, we computed the normal directions and perturbed the matrices in this direction with distances $10^{-5}$ up to $10^{-1}$. For each distance, we again performed our algorithm 10 times and reported the approximation successful if the reconstructed matrix has an error of $10^{-6}$. As before, we considered the reconstruction of the factor matrix successful if the factor matrices exhibited an error of less than $10^{-3}$. If at the beginning we successfully reconstructed the {\em perturbed} matrix, we took this as a sign that we moved into the interior of the cone and thus changed directions for the following steps.

One can see that the algorithm is able to find the closest matrix in $\CP$ in many cases. Peculiarly, it also finds the correct factorization more often than in the reconstruction problem. This is consistent with all our experiments and we suspect that it happens because the alternative solutions discussed in Sec.~\ref{ssc:condition} are not local minima of the approximation problem. However, the approximation problem seems to be more difficult and it runs into many local minima. In many cases, we did not find the original matrix over the 10 attempts. This problem seems to get worse with a larger distance from the boundary.

\begin{table}[t]
\centering
\begin{tabu}{ >{\centering\arraybackslash}p{1.3cm} | >{\centering\arraybackslash}p{1.3cm} | >{\centering\arraybackslash}p{1.3cm} | >{\centering\arraybackslash}p{1.3cm} | >{\centering\arraybackslash}p{1.3cm} | >{\centering\arraybackslash}p{1.3cm} | >{\centering\arraybackslash}p{1.3cm} | >{\centering\arraybackslash}p{1.3cm} | >{\centering\arraybackslash}p{1.3cm} | >{\centering\arraybackslash}p{1.3cm} }
\toprule
\multicolumn{2}{c|}{$10^{-5}$} & \multicolumn{2}{c|}{$10^{-4}$} & \multicolumn{2}{c|}{$10^{-3}$} & \multicolumn{2}{c|}{$10^{-2}$} & \multicolumn{2}{c}{$10^{-1}$} \\ 
a) & b) & a) & b) & a) & b) & a) & b) & a) & b) \\
\midrule
87 & 66 & 87 & 71 & 74 & 66 & 63 & 57 & 66 & 62 \\
\bottomrule
\end{tabu}
\vspace*{.2cm} 
\caption{Success rate of the approximation out of 100 tries for the computation of any factorization (a) and for the original factorization (b) for different distances from the boundary.}
\label{tab:approx}
\end{table} 

\subsection{Numerical observations on the boundary $\partial \CRF \cap \CP^\circ$}
\label{ssc:numinner}
In the previous sections, we have given an exhaustive description of the boundary $\partial \CP$ of the convex cone $\CP$. Since the maximal cp-rank of a $5 \times 5$ matrix is 6, the only remaining interesting part of the cone is the intersection $\partial \CRF \cap \CP^\circ$, which we briefly discussed in Sec.~\ref{sec:intcp}. A detailed description of this ``interior boundary'' is out of reach even for the case $n = 5$, because it is not derived from the extreme rays of the dual problem, nor is the set $\CRF$ convex. With all current tools at our disposal, it seems like a general description of this set comes down to the combinatorial evaluation of all possible nonreducible zero patterns  as for example done in \cite{Hildebrand}.

Nevertheless, it is possible to make some numerical observations. Given a matrix $A$ of cp-rank 6 (see Sec.~\ref{ssc:cp6}) and running our algorithm with a rank 5, it returns an approximate factorization of cp-rank 5. If subsequent runs with different random initial inputs return the same factorization, one can reasonably conclude that this approximate factorization is in fact unique and thus an element of $\partial \CRF \cap \CP^\circ$. In this fashion, we derived the following interesting example:

We begin with the matrix
\begin{equation*}
A_6 = \begin{pmatrix} 
0.4722\ldots & 0.1493\ldots & 0.0225\ldots & 0.1083\ldots & 0.0296\ldots\\
0.1493\ldots & 0.3519\ldots & 0.1442\ldots & 0.0111\ldots & 0.1316\ldots \\
0.0225\ldots & 0.1442\ldots & 0.4121\ldots & 0.2120\ldots & 0.0157\ldots \\
0.1083\ldots & 0.0111\ldots & 0.2120\ldots & 0.2113\ldots & 0.0719\ldots \\
0.0296\ldots & 0.1316\ldots & 0.0157\ldots & 0.0719\ldots & 0.4366\ldots
\end{pmatrix}.
\end{equation*}
This matrix has cp-rank 6. Several runs of our algorithm with rank 5 reveal the best cp-rank 5 approximation $A_5$, which has the factorization
\begin{equation*}
B_5 = \begin{pmatrix} 
0 & 0.6216\ldots & 0.2872\ldots & 0.0578\ldots & 0 \\
0 & 0 & 0.4455\ldots & 0.3692\ldots & 0.1309\ldots \\
0.5097\ldots & 0 & 0 & 0.3904\ldots & 0 \\
0.4158\ldots & 0.1742\ldots & 0 & 0 & 0.0902\ldots \\
0.0310\ldots & 0 & 0.1036\ldots & 0 & 0.6519\ldots
\end{pmatrix}.
\end{equation*}
The matrix $A_5$ is an element of $\partial \CRF \cap \CP^\circ$ and its factor matrix $B_5$ has 11 zeros. This was consistent throughout all our experiments with $\partial \CRF \cap \CP^\circ$. But since this is a hypersurface of dimension 14, there can be no additional polynomial equation that defines it, meaning that the other entries of $B_5$ can be altered freely and we will always get a matrix in $\partial \CRF \cap \CP^\circ$ (up to semialgebraic equations, i.e., inside some possibly small intervals).

In order to verify this observation, we can round the entries of $B_5$ to 2 decimals:
\begin{equation*}
\tilde B_5 = \begin{pmatrix} 
0 & 0.62 & 0.29 & 0.06 & 0 \\
0 & 0 & 0.45 & 0.37 & 0.13 \\
0.51 & 0 & 0 & 0.39 & 0 \\
0.42 & 0.17 & 0 & 0 & 0.09 \\
0.03 & 0 & 0.10 & 0 & 0.65
\end{pmatrix}.
\end{equation*}
Applying our algorithm to $\tilde A_5 = \tilde B_5 \tilde B_5^{\mathsf T}$ recovers the factorization $\tilde B_5$. This strongly suggests that also $A_5 \in \partial \CRF \cap \CP^\circ$ and that this set is in fact fully described by the zero patterns of the factor matrices.
\begin{rem}
Since the number of nonnegative factorizations of elements in $\partial \CP \cap \DN^\circ$ and $\partial\CRF\cap \DN^\circ$ is finite, all these factorizations are locally rigid in the sense of \cite{Krone2021}. One may wonder if they are also infinitesimally rigid. One of the main results in Sec.~6 of the article is that any infinitesimally rigid factorization must have at least 11 entries equal to zero. 
This implies that matrices in $\partial \CP \cap \DN^\circ$ do not admit infinitesimally rigid nonnegative factorizations. The situation in $\CRF$ is different: Some components of the boundary may actually consist of matrices that have a factorization with 11 zeros. As explained above, we have some candidates for such components. However, the examples we found fail to be infinitesimally rigid.

\end{rem}

\section{Open questions and outlook}

Before we discuss the conclusions that we derive from our results, we formulate some of the open problems that we encountered on our way.

The first problem relates to the locus $\partial \CRF \cap \CP^\circ$, i.e., those matrices that have cp-rank 5 but cp$^+$-rank 6. Because our evidence is the result of a numerical approximation algorithm that is also known to be fallible (as we will state next), we refrain from calling this a conjecture:
\begin{prob}
What are the polynomial inequalities that describe $\partial \CRF \cap \CP^\circ$? Are they linear in the entries of the factor matrices, as we suspect according to Sec.~\ref{ssc:numinner}? 
\end{prob}
Corollary 5.8 of \cite{Bomze2015} implies that the number of nonnegative square factorizations of any matrix in $\CRF \cap \CP^\circ$ is finite. However, since neither $\CRF$ nor $\mathcal{CPR}_5(6)$ are convex, answering this question requires a substantially different technique.

The next open problem concerns the numerical ill-posedness of the reconstruction problem. We know by our experiments that the factorization of matrices near the boundary is very chaotic: small changes in the matrix $A$ can result in very large changes in the entries of the factor matrix.
\begin{prob}
For any $\varepsilon > 0$, do there always exist matrices $A,\tilde A \in \partial \CP$ with $\| A - \tilde A \| < \varepsilon$, but, say $\| B B^{\mathsf T} - \tilde B \tilde B^{\mathsf T} \| > 10^{-3}$ for the factor matrices? In other words, can the numerical ill-posedness get arbitrarily bad or are we safe if we allow for a minimal numerical accuracy?
\end{prob}

Notice that if $\mathcal M := \{ X \in \mathbb R^{5 \times 5} : A = X X^{\mathsf T} \}$ denotes the set of all $5\times 5$ possibly negative factor matrices of $A$, the orthogonal group acts on $\mathcal M$ by right multiplication. Any matrix in $
\partial\CP\cap\DN^\circ$ has exactly 5 nonnegative factorizations that correspond to one orbit of the restricted action by the group of $5\times 5$ permutation matrices. The question above pertains to the notion of how close these orbits get to other parts of the boundary, i.e., can the orbits of some matrices of fixed size get arbitrarily close to the boundary at a point far from an actual factorization?

Finally, we address a problem that has been open for some time (see \cite[Section 4.1]{BermanDur2015}) concerning rational factorizations. There are at least two notions for rational factorizations, which are somewhat equivalent, but subtle. Say that a matrix $A\in \CPP_n$ with rational entries admits a rational factorization if there is an integer number $k$, rational numbers $q_1,\dots, q_k$ and rational column vectors $v_1, \dots, v_k$ such that $A = \sum_{j=1}^k q_j \, v_jv_j^{\mathsf T}$. Notice that when writing this expression as a factorization the entries may not be rational, as one has to use the square root of $q_j$ to distribute it among the two vectors. However this turns out be equivalent (see \cite[Section 4]{BermanShaked2018}) to the existence of a rational nonnegative $n \times m$ matrix $C$, such that $A=CC^{\mathsf T}$. We remark that the minimal possible values for $k$ and $m$, in case they exist, may be different and larger than the cp-rank. 
\begin{prob}
Let $A \in \partial \CPP_5\cap \DN^\circ$ have rational entries. Are the square factorizations of $A$ rational in either sense? 
\end{prob}

Given the parametrization of the boundary and the uniquenes of $5\times 5$ factorizations and the parametrization of the algebraic components of the boundary we could hope to answer this question entirely in the $5\times 5$ case. However, this simplification of the problem leads to finding rational points on a variety, which is a priory a difficult question. 

We now conclude our paper with a discussion of the implications and the future outlook. The algebraic description of the boundary $\partial \CP$ gives us an explanation for the complicatedness of the boundary: the high degree makes it very hard to access. Nonetheless, the algebra allows to systematically construct exact matrices in the boundary which seems to be a useful fact. It furthermore suggests that a description with the same level of precision for $n\ge 6$ is likely hopeless. 

It is shown in \cite{AfoninHildebrandDickinson2020} that the nontrival extreme rays of the cone of $6\times 6$ copositive matrices come in $42$ types. In other words, instead of having to deal with two loci, an analogous analysis of that cone would involve $42$ varieties, most of which are expected to yield varieties of much higher degree than the Hildebrand locus. In short, as dimension grows, the complexity of the boundary increases in two directions simultaneously: the number of algebraic components will diverge and each of the resulting varieties will become harder. 

Furthermore, we have so far not found a description of the part $\partial \CRF \cap \CP^\circ$ and we suspect that a full derivation would come down to figuring out all nonreducible zero patterns, which would be feasible in principle. But if the number of zeros is smaller than 11, one also needs to find the associated algebraic equations. This is a difficult algebraic problem. To make things even harder, the lack of convexity of the cones involved makes the description even harder. 

\vspace{0.5cm}

\noindent{\textbf{Acknowledgements:}} We would like to thank Bernd Sturmfels for suggesting this project and for many interesting conversations. We also thank Sascha Timme and Paul Breiding for help with the computation of the degree of $V_{Hi}$ using HomotopyContinuation. This project started while both authors were employed at the Max Planck Institute for Mathematics in the Sciences. M.P.~was partially funded by the Deutsche Forschungsgemeinschaft (DFG, German Research Foundation) – Projektnummer 448293816. 

\bibliographystyle{plain}
\bibliography{BIBLIO}

\end{document}